\documentclass[11 pt, twoside]{amsart}\usepackage{srollens-en}[2011/03/03]
\usepackage[paper = a4paper, left = 3.0cm, right = 3.0cm, headsep = 6mm,
footskip = 10mm, top = 34mm, bottom = 34mm, footnotesep=5mm, headheight =
2cm]{geometry}

\usepackage{color}
\usepackage[T1]{fontenc}
\usepackage{mathrsfs}
\usepackage{mathpazo}
\usepackage{textcomp}

\usepackage{tikz}
\usetikzlibrary{backgrounds}

\usepackage[hypertexnames=false,backref=page,pdftex,
	pdfpagemode=UseNone,
	breaklinks=true,
	extension=pdf,
	colorlinks=true,
	linkcolor=blue,
	citecolor=blue,
	urlcolor=blue,
]{hyperref}

\renewcommand{\epsilon}{\varepsilon}
\renewcommand{\phi}{\varphi}

\title{Lagrangian fibrations on hyperk\"ahler fourfolds}

\author{Daniel Greb}
\address{Daniel Greb\\Institut f\"ur Mathematik\\Abteilung f\"ur Reine Mathematik\\Albert-Ludwigs-Universit\"at Freiburg\\Eckerstr. 1\\79104 Freiburg im Breisgau\\Germany}
\email{daniel.greb@math.uni-freiburg.de}

\author{Christian Lehn}
\address{Christian Lehn\\Institut de Recherche Math\'ematique  
Avanc\'ee\\ Universit\'e de Strasbourg\\
7 rue Ren\'e Descartes\\67084 Strasbourg Cedex\\France}
\email{lehn@math.unistra.fr}

\author{S\"onke Rollenske}
\address{S\"onke Rollenske\\Fakult\"at f\"ur Mathematik\\Universt\"at Bielefeld\\Universit\"atsstr. 25\\33615 Bielefeld\\Germany}
\email{rollenske@math.uni-bielefeld.de}

\begin{document}

\begin{abstract}
Answering the strong form of a question posed by Beauville, we give a short geometric proof that any hyperk\"ahler fourfold containing a Lagrangian subtorus $L$ admits a holomorphic Lagrangian fibration with fibre $L$.
\end{abstract}

\subjclass[2010]{53C26, 14D06, 14E30, 32G10, 32G05.}
\keywords{hyperk\"ahler manifold, Lagrangian fibration}
\maketitle
\section{Introduction}
Let $X$ be a hyperk\"ahler manifold, that is, a compact, simply-connected K\"ahler manifold  $X$  such that $H^0(X, \Omega^2_X)$ is spanned by a holomorphic symplectic form  $\sigma$. By work of Matsushita it is well-known that the only possible non-trivial holomorphic maps from $X$ to a lower-dimensional complex space are Lagrangian fibrations, see section \ref{sec: prelim}. Moreover, a special version of the  so-called Hyperk\"ahler SYZ-conjecture asserts that any hyperk\"ahler manifold can be deformed to a hyperk\"ahler manifold admitting a Lagrangian fibration.

Hence, it is an important problem to find geometric conditions on a given hyperk\"ahler manifold that guarantee the existence of a Lagrangian fibration; here we address a question  posed by Beauville \cite[Sect.~1.6]{beauville10}:
\begin{custom}[Question B]
Let $X$ be a hyperk\"ahler manifold and $L$ a Lagrangian torus in $X$. Is $L$ a fibre of a (meromorphic) Lagrangian fibration $f\colon X\to B$?
\end{custom}
In our previous article \cite{glr11a} it is shown that Question B has a positive answer in case $X$ is non-projective. Moreover, for any hyperk\"ahler manifold that admits an almost holomorphic Lagrangian fibration, a further hyperk\"ahler manfold, birational to the first one, is found, on which the Lagrangian fibration becomes holomorphic.

The approach to the projective case of Beauville's question pursued here is based on a detailed study of the deformation theory of $L$ in $X$. For this, consider the component $\gothB$ of the Barlet space that contains $[L]$ together with its universal family and the evaluation map to $X$:
\begin{equation*}
\begin{gathered}
\xymatrix{ \gothU \ar[r]^\epsilon\ar[d]_\pi & X\\  \gothB.}
\end{gathered}
\end{equation*}
It was shown in \cite[Lemma 3.1]{glr11a} that $\epsilon$ is surjective and generically finite, and that $X$ admits an almost holomorphic Lagrangian fibration if and only if $\deg(\epsilon) = 1$.

If the degree of $\epsilon$ is strictly bigger than one, some deformations of $L$ intersect $L$ in unexpected ways. In order to deal with this, we introduce the notion of \emph{$L$-reduction}: for each projective hyperk\"ahler manifold containing a Lagrangian torus there exists a projective variety $\gothT$ and a rational map $\phi_L\colon X\dasharrow \gothT$, uniquely defined up to birational equivalence, whose fibre through a general point $x$ coincides with the connected component of the intersection of all deformations of $L$ through $x$.
In this situation, we say that $X$ is \emph{$L$-separable} if $\phi_L$ is birational, and prove the following result:
\begin{custom}[Theorem \ref{thm:Lsep}]
 Let $X$ be a projective  hyperk\"ahler manifold and $L\subset X$ a
Lagrangian subtorus.
Then $X$ admits an almost holomorphic fibration with strong fibre $L$ if and only if $X$ is not $L$-separable .
\end{custom}

If $X$ is a hyperk\"ahler fourfold, then we can exclude the case that $X$ is $L$-separable by symplectic linear algebra. Moreover, based upon the rather explicit knowledge about the birational geometry of hyperk\"ahler fourfolds we obtain a positive answer to the strongest form of Beauville's question:
\begin{custom}[Theorem \ref{thm: dim4 holom}]
 Let $X$ be a four-dimensional hyperk\"ahler manifold containing a Lagrangian torus $L$. Then $X$ admits a holomorphic Lagrangian fibration with fibre $L$.
\end{custom}
At the Moscow conference ''Geometric structures on complex manifolds'' Ekaterina Amerik brought to our attention that she had independently shown a related result, based on an observation from \cite{ame-cam08}, to the effect that in dimension four every projective hyperk\"ahler manifold containing a Lagrangian subtorus $L$ admits an almost holmorphic Lagrangian fibration with fibre $L$~\cite{amerik11}.\footnote{After this article was written, Jun-Muk Hwang and Richard Weiss have posted a proof of the projective case of the \emph{weak form} of Beauville's question, producing an \emph{almost holomorphic} Lagrangian fibration on any projective $2n$-dimensional hyperk\"ahler manifold containing a Lagrangian torus, see \cite{HwangWeiss}. Their argument has two parts: one is geometric and one is concerned with abstract group theory.
In contrast, our answer to the \emph{strong form} of Beauville's question, Theorem~\ref{thm: dim4 holom}, is purely geometric, uses global arguments in addtion to local deformation theoretic ones, and uses symplectic linear algebra in place of their group-theoretic arguments. }

\subsection*{Acknowledgements}
The authors want to thank Daniel Huybrechts for his interest in our work and for several stimulating discussions. We are grateful to Ekaterina Amerik for communicating to us the observation contained in Lemma~\ref{lem:symplectic la}, which greatly simplified our previous argument. The second author thanks Laurent Manivel for stimulating discussions, in particular, for pointing out Remark \ref{remark lagrangian grassmanian}. The third author thanks Misha Verbitsky for an invitation to Moscow.

The support of the DFG
through the SFB/TR 45, Forschergruppe 790, and the third author's Emmy-Noether
project was invaluable for the success of the collaboration. The first author
gratefully acknowledges the support of the
Baden-W\"urttemberg-Stiftung via the ``Eliteprogramm f\"ur
Postdoktorandinnen und Postdoktoranden''. The second author acknowledges the support by the CNRS and the Institut Fourier.

\section{Preliminaries and setup of notation}\label{sec: prelim}

\subsection{Lagrangian fibrations}

\begin{defin}
Let $X$ be a hyperk\"ahler manifold. A \emph{Lagrangian fibration} on $X$ is a holomorphic map $f\colon X \to B$ with connected fibres onto a normal complex space $B$ such that every irreducible component of the reduction of every fibre of $f$ is a Lagrangian subvariety of $X$.
\end{defin}
Due to fundamental results of Matsushita it is known that any fibration on a hyperk\"ahler manifold is automatically Lagrangian:
\begin{theo}[\cite{matsushita99, matsushita00,matsushita01, matsushita03}]\label{thm: matsushita}
Let $X$ be a hyperk\"ahler manifold of dimension $2n$. If $f\colon  X \to B$ is a morphism with connected fibres to a normal complex space $B$ with $0 < \dim B < \dim X$, then $f$ is a Lagrangian fibration. In particular, $f$ is equidimensional and $\dim B = n$. Furthermore, every smooth fibre of $f$ is a complex torus.
\end{theo}

\subsection{Meromorphic maps}
Let $X$ be a normal complex space, $Y$ a compact complex space, and $f\colon X \dasharrow Y$ a meromorphic map. Let
\begin{equation}\label{eq:meromoresolution}
\begin{gathered}
\xymatrix{ &\widetilde X\ar[dl]_p \ar[dr]^{\widetilde f}\\ X\ar@{-->}[rr]^f && Y}
\end{gathered}
\end{equation}
be a resolution of the indeterminacies of $f$. The \emph{fibre} $F_y$ of $f$ over a point $y\in Y$ is defined to be $F_y:= p (\widetilde f^{-1}(y))$. This is independent of the chosen resolution.

Recall that a meromorphic map $f\colon X \dasharrow Y$ as above is called \emph{almost holomorphic} if there is a Zariski-open subset $U\subset Y$ such that the restriction $f\restr{\inverse f (U)}\colon\inverse f (U) \to U$ is holomorphic and proper. A \emph{strong fibre} of an almost holomorphic map $f$ is a fibre of $f\restr{\inverse f (U)}$.

Let $X$ be a normal algebraic variety, $B$ a complete algebraic variety, and $f\colon X \dasharrow B$ an almost holomorphic rational map. If $A$ is a divisor on $B$, then its \emph{pullback} via $f$ is defined either geometrically as the closure of the pullback on the locus where $f$ is holomorphic, or on the level of locally free sheaves as $f^*\ko_B(A):=(p_*\widetilde f^* \ko_B(A))^{\vee\vee}$,  where $p\colon\widetilde X\to X$ is a resolution of indeterminacies as in diagram~\eqref{eq:meromoresolution}.

\subsection{Deformations of Lagrangian subtori}\label{sect:deform}
The starting point for our approach to Beau\-ville's question is the deformation theory of a Lagrangian subtorus  $L$ in a hyperk\"ahler manifold $X$. We quickly recall the relevant results from \cite[Sects.~2 and 3]{glr11a}.

The Barlet space $\gothB(X)$ of $X$ (or Chow scheme in the projective setting) parametrises compact cycles in $X$ and it turns out (see \refenum{i} of Lemma~\ref{deflem} below) that there is a unique
 irreducible component $\gothB$ of $\gothB(X)$ containing the point $[L]$. Denoting by $\gothU$ the graph of the universal family over $\gothB$ and by $\Delta$ the \emph{discriminant locus} of $\gothB$, i.e., the set of points parametrising singular elements in the family $\gothB$, we obtain the following diagram.
\begin{equation}\label{diagram:Barlet}
\begin{gathered}
\xymatrix{ \gothU_\Delta \ar@^{(->}[r]\ar[d]& \gothU \ar[r]^\epsilon\ar[d]_\pi & X\\ \Delta \ar@^{(->}[r]& \gothB.}
\end{gathered}
\end{equation}
A detailed analysis of the maps in diagram \eqref{diagram:Barlet} shows that a small \'etale or analytic neighbourhood of $L$ in $X$ fibres over a neighbourhood of $[L]$ in $\gothB$. More precisely, we have the following result.
\begin{lem}[{\cite[Lem.~3.1]{glr11a}}]\label{deflem} Let $X$ be a hyperk\"ahler manifold of dimension $2n$ and let $L$ be a Lagrangian subtorus of $X$. Then, the following holds.
\begin{enumerate}
 \item The Barlet space $\gothB(X)$ is smooth of dimension $n$ near $[L]$. In particular, $[L]$ is contained in a unique irreducible component $\gothB$  of $\gothB(X)$ and  $\gothU$ is smooth of dimension $2n$ near $\pi^{-1}([L])$.
\item The morphism $\epsilon$ is finite \'etale along smooth fibres of $\pi$. In particular, a sufficiently small deformation of $L$ is disjoint from $L$ and  there are no positive-dimensional families of smooth fibres through a general point $x\in X$.
\item If $[L'] \in \gothB$ with smooth $L'$, then $L'$ is a Lagrangian subtorus of $X$.
\end{enumerate}
\end{lem}
\begin{rem}\label{epsilon etale} We remark two simple but useful consequences of Lemma~\ref{deflem}.
\begin{enumerate}
\item The locus $X_\Delta:=\epsilon(\gothU_\Delta)$ is the locus of points $x\in X$ such that there is a singular deformation of $L$ passing through $x$. By dimension reasons it is a proper subset of $X$ and by Lemma \ref{deflem} \refenum{ii} the map $\epsilon$ is finite and \'etale on the preimage of ${X\setminus X_\Delta}$.
\item Statement \refenum{ii} implies in particular that for any two points  $[L],[M] \in \gothB$ the intersection product $[L].[M]$ as cycles in $X$ vanishes. It is therefore impossible for members of the family $\gothB$ to intersect in a finite number of points.
\end{enumerate}
\end{rem}
\subsection{Almost holomorphic Lagrangian fibrations and Barlet spaces.}
The following result relates the deformation theory of $L$ in $X$ discussed above to our question about globally defined almost holomorphic Lagrangian fibrations.
\begin{lem}[{\cite[Lem.~3.2]{glr11a}}]\label{lem:epsilon birational}
Let $X$ be a hyperk\"ahler manifold containing a Lagrangian subtorus $L$. Then $X$ admits an almost holomorphic Lagrangian fibration with strong fibre $L$ if and only if the evaluation map $\epsilon$ in diagram \eqref{diagram:Barlet} is bimeromorphic.
\end{lem}
If $\epsilon$ is birational, then $\pi\circ \inverse \epsilon$ is the desired almost holomorphic fibration (up to normalisation of  $\gothB$). For the other direction one uses the Barlet space of a resolution of indeterminacies.

\section{$L$-reduction and $L$-separable manifolds}
Let $X$ be a projective hyperk\"ahler manifold containing a Lagrangian subtorus $L$. In this section we start our analysis of the maps in the associated diagram~\eqref{diagram:Barlet}. Recall from Lemma~\ref{lem:epsilon birational} above that in order to answer Beauville's question positively we have to show that the evaluation map $\epsilon$ is birational.
\subsection{$L$-reduction}
Here, we construct a meromorphic map associated with the covering family $\{L_t\}_{t \in \gothB}$. Generically, this map is a quotient map for the meromorphic equivalence relation defined by the family $\{L_t\}$, i.e., generically it identifies those points in $X$ that cannot be separated by members of $\{L_t\}$.
\subsubsection{Construction of the $L$-reduction}
We work in the setup summarised in diagram \eqref{diagram:Barlet}. We set $\gothU_{\mathrm reg}:= \epsilon^{-1}(X \setminus X_\Delta)$. Recall from Remark \ref{epsilon etale} that the map
\[\epsilon|_{\gothU_{\mathrm reg}}\colon \gothU_{\mathrm reg} \to X \setminus X_\Delta \]
is finite \'etale; we denote its degree by $d$.

The map $\epsilon|_{\gothU_{\mathrm reg}}$ induces a morphism $X \setminus X_\Delta \to \mathrm{Sym}^d(\gothU_{\mathrm reg})$. Composing this map with the natural morphism $\mathrm{Sym}^d(\gothU_{\mathrm reg}) \to \mathrm{Sym}^d(\gothB)$ induced by $\pi\colon \gothU \to \gothB$, we construct a morphism $X \setminus X_\Delta \to \mathrm{Sym}^d(\gothB)$. This morphism naturally extends to a rational map $\psi\colon X \dasharrow \mathrm{Sym}^d(\gothB)$. Let
$$\begin{xymatrix}{\widetilde X\ar^p[d] \ar^-{\widetilde \psi}[r] & \mathrm{Sym}^d(\gothB)\\
X &
}\end{xymatrix}$$ be a resolution of singularities of the indeterminacies of $\psi$ with $\widetilde X$ nonsingular. The Stein factorisation of $\widetilde \psi$ then yields the following diagram.
\[\begin{xymatrix}{
X\ar@{-->}[dr]_{\phi_L} &\ar[l]_p\widetilde X \ar[rd]^{\widetilde \psi} \ar[d]^{\widetilde \phi}\\
&  \gothT\ar[r] & \mathrm{Sym}^d(\gothB).
}
\end{xymatrix}
\]
Here, $\phi_{L}=\tilde\phi\circ\inverse p\colon X \dasharrow \gothT$ is the rational map induced by $\widetilde \phi$. Noting that $\phi_L: X \dasharrow \gothT$ is unique up to birational equivalence, and hence canonically associated with the pair $(X, L)$, we call it \emph{the $L$-reduction} of $X$.
\begin{rem}\label{remark nice fibre of phi}
For every point $x \in X \setminus X_{\Delta}$ there are exactly $d$ pairwise distinct smooth tori $L_1, \dots, L_d$ in the family $\{L_t\}_{t \in \gothB}$ containg $x$. By construction, $\phi_L$ is defined at $x$ and maps it to 
 the class of $([L_1], \dots, [L_d])$ in $\mathrm{Sym}^d(\gothB)$.
\end{rem}

\subsubsection{First properties of the $L$-reduction}
The following set-theoretical assertion is an immediate consequence of the construction of $\phi_L$.
\begin{lem}\label{lem: F_x}
The fibre of $\phi_L$ through a point  $x\in X \setminus X_\Delta$ coincides with the connected component of
\[
\bigcap_{[M]\in \gothB,\, x \in M} M
\]
containing $x$.
\end{lem}
\begin{proof}
If $x\in X\setminus X_\Delta$, then $\epsilon $ is \'etale in every point of the preimage $\inverse\epsilon(x)$. Thus the image $\pi(\inverse\epsilon(x))=\{[L_1], \dots, [L_d]\}$ consists of the points in $\gothB$ that parametrise the $d$ pairwise distinct subtori in $X$ through $x$. In particular, the meromorphic map $\psi\colon X\dasharrow \mathrm{Sym}^d(\gothB)$ is defined at $x$ and its fibre is
\begin{equation}\label{eq:intersect}
\inverse\psi(\psi(x))=\bigcap_i L_i.
\end{equation}
After taking the Stein factorisation, the fibre of $\phi_L$ is the component of \eqref{eq:intersect} through $x$, as claimed.
\end{proof}

\begin{lem}\label{lem: L-reduction almost holomorphic}
 Let $X$ be a projective hyperk\"ahler manifold containing a Lagrangian subtorus $L$. Then the $L$-reduction $\phi_{L}\colon X \dasharrow \gothT$ is almost holomorphic.
\end{lem}
\begin{proof}
Let $\mathrm{dom}(\phi_L)$ be the domain of definition of $\phi_L$, and let $Z:= X\setminus \mathrm{dom}(\phi_L)$ be the locus where $\phi_L$ is not defined. We have to show that the general fibre of $\phi_L$ does not intersect $Z$.
 
Aiming for a contradiction, suppose that for a general $x_0 \in X \setminus X_\Delta$ the fibre $F_{x_0}$ of $\varphi_L$ through $x_0$ intersects $Z$ nontrivially. Recall from item \refenum{i} of Remark \ref{epsilon etale} that $X_\Delta = \epsilon(\gothU_\Delta)$ is the locus swept out by singular deformations of $L$ and from Remark \ref{remark nice fibre of phi} that $\phi_L$ is holomorphic on $X\setminus X_\Delta$. Take a point $z\in F_{x_0} \cap Z$.
Consider the graph $X'\subset X\times \gothT$ of $\varphi_L$ with projections $p\colon X' \to X$ and $\phi_L'\colon X' \to \gothT$.
As explained for example in \cite[Sect.~1.39]{Debarre}, the closed subset $Z$ can be described as
\begin{equation}\label{eq:center}
Z = \{x \in X \mid \dim p^{-1}(x) >0\}.
\end{equation}
As $X$ is normal and $p$ is birational, $p$ has connected fibres. Thus, the variety $C':=\phi_L'(p^{-1}(z))$ is connected. We list some further properties of $C'$:
\begin{enumerate}
	\item $\dim C' > 0$, because $\dim p^{-1}(z)>0$ and $X'$ is the graph of $\phi_L$,
	\item $\varphi_L(x_0) \in C'$ as $z \in F_{x_0}$,
	\item the point $z$ is contained in all fibres over points in $C'$.
\end{enumerate}

Suppose for the moment that we had a diagram
\[
\xymatrix{
C \ar@{^(->}[r] \ar[d]^{\phi_L\vert_C}& X \ar@{-->}[d]^{\phi_L}\\
C' \ar@{^(->}[r] & \gothT
}
\]
such that $C$ is connected, $x_0 \in C$ and $\phi_L\vert_C$ is a local isomorphism at the point $x_0$.  We claim that this would produce a contradiction. Namely, let $L_1, \dots, L_d$ be the $d=\deg\epsilon$ pairwise distinct tori in the family $\{L_t\}_{t \in \gothB}$ containing $x_0$. Since $C$ is connected and $C \not \subset F_{x_0}$ by item \refenum{i} above, Lemma~\ref{lem: F_x} implies that there exists $k \in \{1, \dots, d\}$ such that $C \not \subset L_k$. By Lemma \ref{deflem}, small deformations of $L_k$ constitute a fibration in an analytic neighbourhood of $x_0$. Thus, for all points $y\in C\setminus \{x_0\}$ sufficiently close to $x_0$ there is a small deformation $L_y$ of $L_k$ with $y\in L_y$ and 
\begin{equation}\label{eq:emptyintersect}
 L_y\cap L_k = \emptyset.
\end{equation}
On the other hand, item \refenum{iii} above and Lemma~\ref{lem: F_x} imply that 
\[
z \in F_y \cap F_{x_0}  \subset L_y \cap L_k,
\]
which in view of \eqref{eq:emptyintersect} is absurd.

It remains to find the variety $C$. We observe that it suffices to construct $C$ in an Euclidean open neighbourhood of $x_0$. Invoking the generality assumption on $x_0$ and the implicit function theorem we find a small neighbourhood $U\ni x_0$ such that the restriction $\phi_L\colon U \to V:=\phi_L(U)$  is a trivial holomorphic fiber bundle. In particular, $V \subset \gothT$ is open and there is a section $C \subset U$ for the subvariety $C' \cap V$. The only remaining property to be fulfilled is connectedness of $C'\cap V$ and $C$. This may be achieved by shrinking $V$ and $U$, and so we conclude the proof.
\end{proof}

\begin{defin}
A projective hyperk\"ahler manifold $X$ containing a Lagrangian subtorus $L$ is called \emph{$L$-separable} if its $L$-reduction $\phi_{L}\colon X \dasharrow \gothT$ is birational.
\end{defin}

\subsection{Lagrangian fibrations on non-$L$-separable manifolds}

\begin{theo}\label{thm:Lsep}
Let $X$ be a projective  hyperk\"ahler manifold and $L\subset X$ a
Lagrangian subtorus.
Then $X$ admits an almost holomorphic fibration with strong fibre $L$  if and only if $X$ is not $L$-separable.
\end{theo}
As a consequence of this result we can reformulate Beauville's question in the following way.
\begin{custom}[Question B\textquotesingle\textquotesingle]
 Does there exist a projective hyperk\"ahler manifold $X$ together with a Lagrangian subtorus $L$ such that $X$ is $L$-separable?
\end{custom}
\begin{proof}[Proof of Theorem~\ref{thm:Lsep}]
If $X$ is not $L$-separable, the $L$-reduction $\phi_L\colon X \dasharrow \gothT$ is an almost holomorphic map (Lemma \ref{lem: L-reduction almost holomorphic}) such that $0<\dim \gothT < \dim X$. Thus by \cite[Thm.~6.7]{glr11a}, the map $\phi_L$ is an almost holomorphic Lagrangian fibration on $X$. By the description of the general fibre of the $L$-reduction (Lemma \ref{lem: F_x}), the torus $L$ is a strong fibre of $\phi_L$.

If conversely $f\colon X\dasharrow B $ is an almost holomorphic Lagrangian fibration with strong fibre $L$, then through the general point there is a unique Lagrangian subtorus in $\gothB$ and the $L$-reduction coincides with the rational map $\pi\circ\inverse \epsilon\colon X\dashrightarrow \gothB$. In particular, $X$ is not $L$-separable.
\end{proof}

\section{Intersections of Lagrangian subtori} 
As before, let $X$ be a projective hyperk\"ahler manifold containing a Lagrangian sub\-torus $L$. In this section we study a neighbourhood of $L$ in $X$ more closely, which leads to several results about the geometry of intersections of different members in the family $\gothB$ of deformations of $L$. We are going to use the notation and the results of Section~\ref{sect:deform} throughout.

By Lemma \ref{deflem}, $\gothB$ is smooth at $[L]$ and we can find a neighbourhood $V$ of $[L]$ such that the restriction
$\epsilon\restr\gothU_V\colon  \gothU_V \to X$ of the evaluation map to the preimage $\gothU_V:= \pi^{-1}(V)$ embeds $\gothU_V$ into $X$. We may thus consider $\gothU_V$ as an open subset of $X$. The intersection of $\gothU_V$ with a submanifold  $M\subset X$ is depicted in Figure \ref{fig: M cap L}.

\begin{figure}[ht] 
\begin{tikzpicture}
[
scale=.5,
outside/.style = {black!30},
rand/.style = {black!60},
>=stealth
]
\tikzstyle{every node}=[font=\tiny]

\clip (-5,-1.1) rectangle (5,12.3);

\begin{scope}[rand]
\draw (0,11) ellipse (2 and 1);
\draw (-2,4) arc (180:360:2 and 1);
\draw[outside, densely  dotted] (2,4)  arc (0:180:2 and 1);
\draw (-2,11) -- (-2,4) ++ (4,0) -- (2,11);
\end{scope}
\node at (-2,11) [above left] {$\gothU_V$};
\draw[thick] (0,11) node [above right] {$L$} -- (0,4);

\draw[rand] (0,0) ellipse (2 and 1);
\draw [->, semithick] (0, 2.8) to node [right]  {$\pi$} (0,1.2) ;
\node (L) at (0,0) [fill,circle,inner sep=1.2pt] {};
\draw (L) node [anchor=north] {$[L]$};
\node at (-2,.6) [above ] {$V$};

\draw[outside] (-3, 8.5) -- (-2,9);
\draw  (-2,9) .. controls +(1,.5) .. ++ (2,0) .. controls +(1,-.5) .. ++ (2,0)  ;
\draw[outside] (2,9) .. controls +(2,1) and (4,8) .. (2,8);
\draw (0,7.5) .. controls +(-1,-.5)  .. +(-2,-.5)  (0,7.5) .. controls +(1,.5) .. +(2,.5);
\draw[outside] (-2,7) .. controls (-4,7) and (-4, 5,5) .. (-2,6);
\draw[dashed]  (-2,6) .. controls +(1,.5) .. ++ (2,0) .. controls +(1,-.5) .. ++ (2,0) node [above left] {$S$};
\draw[outside] (3, 6.5) -- (2,6);

\node at (3,6) {$M$};

\draw[dashed]  (-2,0) .. controls +(1,.5) .. ++ (2,0) .. controls +(1,-.5) .. ++ (2,0) node [below right] {$C$};
\draw (0,0) .. controls +(-1,-.5)  .. (210:2 and 1)  (0,0) .. controls +(1,.5) .. ++ (30:2 and 1);

\end{tikzpicture}
\caption{The neighbourhood $\gothU_V$ of $L$ and its projection to $V\subset \gothB$.}\label{fig: M cap L}
\end{figure}

\begin{lem} \label{lem:smooth intersection} Let $M \subseteq X$ be a smooth and proper submanifold, and $L \subset X$ a smooth
Lagrangian torus that intersects $M$ nontrivially. Then a generic small deformation of $L$ 
 has smooth intersection with $M$.
\end{lem}
\begin{proof} We continue to use the notation introduced above.
Since  $\gothU_V$ is open in $X$, the intersection $M \cap \gothU_V$ is smooth. Furthermore, the map $\pi\restr{M \cap \gothU_V}\colon M \cap \gothU_V \to V$ is proper, because $\pi$ is proper and $M$ is compact. We can therefore apply the theorem on generic smoothness to $\pi\restr{M \cap \gothU_V}$ which proves the result.
\end{proof}

\begin{prop} \label{torus intersection} Let $M\subseteq X$ be a compact submanifold and $L \subseteq X$ be a general Lagrangian subtorus,  such that $L\cap M \neq \emptyset$. Then $N_{L\cap M/M}$ is trivial. If $M$ is a complex torus, then $L\cap M$ is a disjoint union of tori.
\end{prop}

\begin{proof}
As $L$ is general, the intersection $L\cap M$ is smooth by Lemma~\ref{lem:smooth intersection}. Moreover, both statements can be verified by looking at one connected component of $L\cap M$ at a time. We invoke the notation introduced in the beginning of this section, and let $T$ be a connected component of $L\cap M$. If $V$ is sufficently small, then the inclusion $L\cap M \into \gothU_V \cap M$ induces a one-to-one correspondence of their respective connected components. Let $S$ be the unique component of $\gothU_V \cap M$ corresponding to $T$. By generality of $L$ we may assume that $\pi\restr {S}$ is a smooth map, thus  $C:=\pi(S)\subset V$ is smooth of dimension $n-\dim T$ near $[L]$. Moreover, $C$ parametrizes those small deformations of $L$ that induce a flat deformation of $T$ inside $M$. Corresponding to the family $S\to C$ we thus obtain a classyfying map $\chi\colon C \to \mathscr{D}\left(M\right)$ from $C$ to the Douady-space of $M$.

On the level of tangent spaces we have $T_{C}([L])\subset T_\gothB([L])=H^0\bigl(L,\, N_{L/X}\bigr)$, where the last equality comes from the Hilbert-Chow morphism, compare \cite[Lem.~3.1]{glr11a}. The morphism $\chi$ induces a map $\chi_*:T_{C}([L]) \to H^0\bigl(T,\, N_{T/M}\bigr)$. But small deformations of $T$ inside $M$ induced by deformations of $L$ are disjoint from $T$ by Lemma \ref{deflem} \refenum{ii}. Thus the map $\chi_*$ is injective, and the image of $T_{C}([L])$ consists of nowhere vanishing sections. For dimension reasons these sections generate the normal bundle of $T$ in $M$, and consequently $N_{T/M}$ is trivial, as claimed.	

\enlargethispage*{2\baselineskip}
If $M$ is a torus as well, then $T_M\restr{T}$ is likewise trivial. So, by the normal bundle sequence
\[
\xymatrix{
0\ar[r] & T_{T} \ar[r] & T_M\restr{T}\ar[r] & N_{T/M} \ar[r] & 0 \\
}
\]
 the tangent bundle $T_{T}$ is trivial, and thus $T$ is a complex torus.
\end{proof}
Based on the preceeding result we can now refine the observation in Remark~\ref{epsilon etale}(ii):

\begin{lem}\label{lem:elliptic}
 Let $X$ be a four-dimensional hyperk\"ahler manifold. Let $L$ and $M$ be two Lagrangian tori intersecting smoothly, and set $I = L \cap M$.
Then, $I$ is a finite disjoint union of elliptic curves.
\end{lem} 
\begin{proof}It remains to exclude the existence of zero-dimensional connected components of $I$. By general Lagrangian intersection theory, see for example~\cite[Introduction]{BehrendFantechi}, we have
$$[L].[M] = \chi(I). $$
However, this already implies the claim, since by Proposition~\ref{torus intersection} above, any positive dimensional component of $I$ is a smooth elliptic curve, contributing zero to the Euler characteristic $\chi(I)$.
\end{proof}
\begin{cor}\label{cor:degree}
Let $X$ be a four-dimensional projective hyperk\"ahler manifold, $L$ a Lagrangian subtorus. Assume that $X$ is $L$-separable. Then, the evaluation map $\epsilon$ in Diagram~\eqref{diagram:Barlet} has degree at least three.
\end{cor}
\begin{proof}
As $X$ is assumed to be $L$-separable, Theorem \ref{thm:Lsep} and Lemma \ref{lem:epsilon birational} imply that $\epsilon$ is not birational. It remains to exclude the case $\deg \epsilon =2$. By Lemma \ref{lem: F_x}, the $L$-separability means that at a general point $x \in X$ the connected component of $\bigcap_{[M]\in \gothB,\, x \in M} M$ is just $x$. If $\deg\epsilon = 2$ there are just two tori in $\gothB$ containing $x$, say $M_1$ and $M_2$. As $x$ was general, Lemma \ref{lem:smooth intersection} tells us that $M_1 \cap M_2$ is smooth. Then Lemma \ref{lem:elliptic} contradicts the fact that the connected component of $M_1 \cap M_2$ containing $x$ is $\{x\}$.
\end{proof}

\section{Hyperk\"ahler fourfolds}
Using the results from the last section we can now prove our main result which gives the strongest possible positve answer to Beauville's question:
\begin{theo}\label{thm: dim4 holom}
Let $X$ be a four-dimensional hyperk\"ahler manifold containing a Lagrangian torus $L$. Then $X$ admits a holomorphic Lagrangian fibration with fibre $L$.
\end{theo}
\begin{rem}
We are grateful to E.~Amerik for communicating the following linear algebra observation to us which serves to exclude $L$-separable manifolds $X \supset L$ in dimension four. This greatly simplified a previous deformation-theoretic argument.
\end{rem}
\begin{lem}\label{lem:symplectic la}
 Let $V$ be a four-dimensional symplectic vector space with symplectic form $\sigma$, and let $W_1, W_2, W_3 \subset V$ be three Lagrangian subspaces satisfying $\dim W_i\cap W_j = 1$ for all $i\neq j$. Then $W_1\cap W_2 \cap W_3\neq \{0\}$.
\end{lem}

\begin{proof}
 Suppose on the contrary that $W_1\cap W_2 \cap W_3 = \{0\}$ and consider the span $\langle W_1, W_2\rangle$. It is of dimension $3$ as $\dim W_1 \cap W_2 = 1$. Moreover, we claim that
\begin{equation}  \label{eq:inclusion}                                                                                                                                                      
   W_3 \subset \langle W_1, W_2\rangle.                                                                                                                                                                   \end{equation}
 Indeed, otherwise we would have $\dim W_3 \cap \langle W_1, W_2\rangle = 1$, implying that the intersections $ W_3 \cap \langle W_1, W_2\rangle = W_3 \cap W_1= W_3 \cap W_2$ are all one-dimensional, in contradiction to our assumption that $W_1\cap W_2 \cap W_3= \{0\}$.

Now, again using $W_1\cap W_2 \cap W_3 = \{0\}$ we write 
$ \langle W_1, W_2\rangle = W_3 \oplus (W_1 \cap W_2)$.
 As $V$ is symplectic and $W_3$ is Lagrangian, there is $v\in W_1 \cap W_2$ and $w \in W_3$ such that $\sigma(v,w)\neq 0$. According to the inclusion \eqref{eq:inclusion} we can write $w=w_1+w_2$ with $w_i \in W_i$, so that $$0 \neq \sigma(v,w)= \sigma(v,w_1) + \sigma(v,w_2)= 0 +0=0 ,$$ as $W_1$ and $W_2$ are Lagrangian. Contradiction.
\end{proof}

\begin{rem}\label{remark lagrangian grassmanian}
Lemma~\ref{lem:symplectic la} can also be proven using the following beautiful geometric argument which was explained to us by Laurent Manivel:  The Grassmanian of Lagrangian subspaces in $V\isom \IC^4$ is biholomorphic to the (smooth) intersection of the Pl\"ucker quadric $\widetilde Q \subset \mathbb{P}(\bigwedge^2 V)$ and the linear subspace defined by vanishing of the symplectic form $\sigma:~\bigwedge^2 V \to \mathbb{C}$, and hence is a smooth quadric $Q \subset \IP^4$. The condition $\dim W_i\cap W_j = 1$ means that the line in $\IP^4$ joining the points $[W_i], [W_j] \in Q$ is contained in $Q$. If the triple intersection $W_1\cap W_2 \cap W_3 = \{0\}$, then $[W_1], [W_2], [W_3]$ span a plane $P$. But $P \cap Q$ has degree $2$ and thus cannot contain a union of $3$ lines.
\end{rem}

\begin{prop}\label{prop: not L-sep}
Let $X$ be a four-dimensional projective hyperk\"ahler manifold containing a Lagrangian torus $L$. Then $X$ is not $L$-separable.
\end{prop}
\begin{proof}
Suppose on the contrary that $X$ is $L$-separable. Given a general point $x \in X$, it follows from Lemma~\ref{lem: F_x} and Corollary~\ref{cor:degree} that there exists a natural number $d \geq 3$, and $d$ smooth Lagrangian subtori that locally cut out $x$.
The point $x$ being general,  Lemma \ref{lem:smooth intersection} and Lemma~\ref{lem:elliptic} imply that there exist two such tori, say $L_1$ and $L_2$, that intersect in an elliptic curve $E$ at $x$. Any other torus in $\gothB$ passing through $x$ either contains $E$ or cuts out a zero-dimensional subscheme. 

As a consequence of $L$-separability there exists a lagrangian torus $L_3$ containing the point $x$ but not containing $E$, such that the intersection scheme $L_1\cap L_2\cap L_3$ is zero-dimensional at the point $x$.
Again invoking that $x$ was general, we may assume that the intersections $L_1 \cap L_3 $, $L_2 \cap L_3$ and $L_1 \cap L_2 \cap L_3$ are smooth at $x$.
Consequently, the three Lagrangian subspaces $W_i:= T_{L_i, x} \subset T_{X,x}$ satisfy the assumptions of Lemma~\ref{lem:symplectic la}. It follows that $W_1 \cap W_2 \cap W_3 \neq \{0\}$, contradicting our choice of $L_1, L_2, L_3$. Therefore, $X$ cannot be $L$-separable.
 \end{proof}
 \begin{proof}[Proof of Theorem \ref{thm: dim4 holom}]
  If $X$ is not projective, we are done by \cite[Thm.~4.1]{glr11a}, so we may assume $X$ to be projective. By Proposition~\ref{prop: not L-sep}, $X$ is not $L$-separable and hence admits an almost holomorphic Lagrangian fibration $f\colon X\dashrightarrow B$ by Theorem~\ref{thm:Lsep}. It remains to show that the existence of an almost holomorphic Lagrangian fibration implies the existence of a holomorphic one, which will be done in Lemma~\ref{lem: already holomorphic} below.
 \end{proof}

\begin{lem}\label{lem: already holomorphic}
 Let $f\colon X\dashrightarrow B$ be an almost holomorphic Lagrangian fibration on a projective hyperk\"ahler fourfold. Then there exists a birational modification $\psi\colon  B\dashrightarrow B'$ such that $\psi\circ f\colon X\to B'$ is a holomorphic Lagrangian fibration.
\end{lem}

The proof of Lemma~\ref{lem: already holomorphic} rests on the explicit knowledge of the birational geometry of hyperk\"ahler fourfolds. For this we recall the notion of \emph{Mukai flop}: Assume that a hyperk\"ahler fourfold $X$ contains a smooth subvariety $P\isom \IP^2$. If we blow up $P$, the exceptional divisor is isomorphic to the projective bundle  $\IP(\Omega_{\IP^2}^1)$, and it is well known that it can be blown down in the other direction to yield another hyperk\"ahler manifold $X'$. The resulting birational transformation $X\dashrightarrow X'$ is called the \emph{Mukai flop} at $P$.

\begin{proof}[Proof of Lemma~\ref{lem: already holomorphic}]
 By \cite[6.2]{glr11a} there exists a holomorphic model for $f$, that is, a Lagrangian fibration $f'\colon X'\to B'$ on a possibly different hyperk\"ahler manifold $X'$ and a diagram
\[ \xymatrix{ X\ar@{-->}[d]_f\ar@{-->}[r]^\phi & X'\ar[d]^{f'}\\ B\ar@{-->}[r]^\psi& B'}\]
with birational horizontal arrows such that $\phi$ is an isomorphism near the general fibre of $f$.

We claim that the composition $f'\circ \phi=\psi\circ f$ is holomorphic and thus a  Lagrangian fibration on $X$. To see this first note that by \cite[Thm.~1.2]{wierzba-wisniewski03} the map $\phi$ factors as a finite composition of Mukai flops, so by induction we may assume that $\inverse \phi$ is the simultaneous Mukai flop of a disjoint union of embedded projective planes $\IP^2\isom P_i\subset X'$.

As $\phi$ is holomorphic near a general fibre of $f'$, none of the $P_i$'s can intersect the general fibre. Thus $f'(P_i)$ is a proper  subset of $B'$ and hence of dimension at most 1. Since there is no non-constant map from $\IP^2$ to a curve, $f'(P_i)$ is a single point. In other words, the locus of indeterminacy of $\inverse \phi$ is contained in the fibres of $f'$, and thus the composition $f'\circ \phi$ remains holomorphic.
\end{proof}

\end{document}